\documentclass[12pt,reqno]{amsart}

\usepackage[utf8]{inputenc}
\usepackage[T1]{fontenc}
\usepackage{amsmath, amssymb, amsfonts, amsthm}
\usepackage{mathrsfs}
\usepackage{enumitem}
\usepackage{hyperref}
\usepackage{graphicx}
\usepackage[
paper=a4paper,
left=27mm,
right=27mm,
top=30mm,
bottom=30mm
]{geometry}
\usepackage{tikz}
\usetikzlibrary{decorations.markings, arrows.meta, calc}
\hypersetup{
	colorlinks = true,
	linkcolor  = blue,
	citecolor  = blue,
	urlcolor   = blue
}

\theoremstyle{plain}
\newtheorem{thm}{Theorem}[section]

\newtheorem{lem}[thm]{Lemma}

\theoremstyle{definition}

\theoremstyle{remark}

\title{Regular automorphisms of periodic $p'$-groups need not be splitting}

\author{V.\,S.~Atabekyan}
\address{Department of Mathematics, Yerevan State University, Armenia}
\email{avarujan@ysu.am}

\date{\today}


\begin{document}

	\maketitle
	
	\begin{abstract}
		
		For a fixed prime $p$, a group is called a $p'$-group if the orders of all its elements are coprime to $p$. It is well known that in a finite $p'$-group an automorphism of order $p$ is regular if and only if it is a splitting automorphism of period $p$. In the Kourovka Notebook, V.\,D.~Mazurov asked whether these two notions coincide for periodic $p'$-groups. We answer this question in the negative by constructing a periodic group of prime period $q$ that admits a regular automorphism of order $p \neq q$ which is not splitting.
	\end{abstract}
	
	\section{Introduction}
	
		An automorphism $\phi$ of a group $G$ is called \textit{regular} (or \textit{fixed-point-free}) if $C_G(\phi) = \{1\}$. It is called \textit{splitting of period $p$} if $\phi^p = 1$ and $$g\,g^{\phi}g^{\phi^2}\cdots g^{\phi^{p-1}} = 1$$ for all $g \in G$. In the Kourovka Notebook \cite{KN} (Problem 17.71(b)), V.\,D.~Mazurov posed the following question: \textit{Let $\phi$ be an automorphism of order $p$ without nontrivial fixed points of a periodic group $G$ containing no nontrivial $p$-elements (that is, elements whose order is a power of the prime $p$). Must $\phi$ be splitting?}
	
	For finite groups the answer to Mazurov's question is affirmative: it is well known that every regular automorphism of order $p$ is a splitting automorphism of period $p$ (see, for example, \cite{DG}). Conversely, if $G$ has no nontrivial $p$-elements, then every splitting automorphism of order $p$ is automatically regular (see, for example, \cite[Lemma~7.1.3]{Kh}). In this paper we show that for infinite periodic groups a regular automorphism of order $p$ need not be splitting, even when the group has no nontrivial $p$-elements.
	
	Regular automorphisms have a strong influence on the structure of a group. For instance, by Thompson's theorem, a finite group admitting a regular automorphism of prime order is nilpotent \cite{Th}; more generally, a finite group admitting a fixed-point-free group of automorphisms is solvable \cite{R}.
	
	Clearly, if $G$ is a periodic group of period $p$, then every inner automorphism $i_a \in \mathrm{Inn}(G)$ ($i_a(b) = a^{-1}ba$) of order $p$ is splitting of period $p$. The converse does not hold in general, but it does for free Burnside groups $B(m,p)$ of prime period $p > 1001$. Recall that the free Burnside group $B(m,p)$ of period $p$ and rank $m$ is the quotient of the absolutely free group $F_m$ by the normal subgroup $F_m^p$ generated by all $p$th powers of elements of $F_m$. Answering a question of S.~Ivanov from \cite{KN}, it was proved in \cite{At14} that if the order of a splitting automorphism $\phi$ of $B(m,n)$ is a prime power (where $n \ge 1003$ is an odd period), then $\phi$ is inner (see also \cite{At13ms}). We also note that for every automorphism of order $2$ of an arbitrary nonabelian periodic group of odd period, nontrivial fixed points were constructed explicitly in \cite{AtAs19}.
	
	The negative answer to Mazurov's question follows from the next theorem.
	
	\begin{thm}\label{th}
		For any two distinct primes $p\ge3$ and $q>10^{10}$ there exists a periodic group $G$ with two generators $a,b$ and period $pq$ such that:
		\begin{enumerate}
			\item the generator $a$ has order $p$;
			\item the normal closure $\langle b\rangle^G$ of the element $b$ in $G$ is a periodic group of period $q$;
			\item the inner automorphism $i_{a}$ acts on the subgroup $\langle b\rangle^G$ without nontrivial fixed points and is not a splitting automorphism of $\langle b\rangle^G$.
		\end{enumerate}
	\end{thm}

	The group in Theorem \ref{th} is constructed in Section \ref{p2}. Sections \ref{p3}, \ref{p4} and \ref{p5} establish the properties of this group that we need, and the proof of Theorem \ref{th} is completed in Section \ref{p6}.
	
	\section{Construction of the periodic group}\label{p2}
	In this section we construct a periodic group satisfying the conditions of Theorem \ref{th}, where $q$ is a sufficiently large prime, say $q>10^{10}$, and $p\ge3$ is an arbitrary prime distinct from $q$. The group is built using the strictly Burnside $q$-product of groups introduced by S.~Ivanov \cite{I}. This construction is also presented in Ol'shanskii's monograph \cite[Ch.~11, \S36, item~2]{O}, where the strictly Burnside $q$-product of a family of groups $G_\mu$, $\mu\in I$, is denoted by $\underset{\mu\in I}{\bigcirc}^q G_\mu$. For the reader's convenience, we state the definition of the product $\bigcirc^q$ for the family we need, namely the strictly Burnside $q$-product of two cyclic groups of orders $p$ and $q$. We follow the notation of \cite[Ch.~11, \S36, item~2]{O}.
	
	Let $F=\langle a\rangle_p\ast \langle b\rangle_q$ be the free product of the cyclic group $\langle a\rangle_p$ of order $p$ generated by $a$ and the cyclic group $\langle b\rangle_q$ of order $q$ generated by $b$, and let $C$ be the Cartesian subgroup of $F$. Every nonidentity element of $F$ has a uniquely determined normal form, which is a word over the alphabet $(\langle a\rangle_p\setminus\{1\})\cup (\langle b\rangle_q\setminus\{1\})$. Let $N\subset F$ be the subgroup generated by all words $x^q$ over this alphabet that represent an element of $C$. In other words, $x^q\in N$ if and only if
	$x^q\in \mathrm{Ker}(\gamma)$, where
	\begin{equation}\label{gamma}
		\gamma:F=\langle a\rangle_p\ast \langle b\rangle_q\rightarrow \langle a\rangle_p\times \langle b\rangle_q=F/C
	\end{equation}
	is the canonical homomorphism from the free product onto the direct product of $\langle a\rangle_p$ and $\langle b\rangle_q$. By definition, $\langle a\rangle_p\bigcirc^q \langle b\rangle_q$ is the quotient group $F/N$: $$\langle a\rangle_p\bigcirc^q \langle b\rangle_q=F/N.$$
	Thus
	\begin{equation}\label{N}
		N=\langle x^q\mid x\in F; x^q\in C\rangle\subset C\cap F^q,
	\end{equation} where $F^q=\langle x^q\mid x\in F\rangle$, and we have the sequence of natural epimorphisms
	\begin{equation}\label{alphabeta}
		F\xrightarrow{\alpha}F/N\xrightarrow{\beta}F/C,
	\end{equation}
	that is,
	$$\langle a\rangle_p\ast \langle b\rangle_q\xrightarrow{\alpha}\langle a\rangle_p\bigcirc^q \langle b\rangle_q \xrightarrow{\beta}\langle a\rangle_p\times \langle b\rangle_q,$$
	with $\beta\circ\alpha=\gamma$.
	\section{Some properties of the group $\langle a\rangle_p\bigcirc^q \langle b\rangle_q$}\label{p3}
	Using the identity $yx=xyy^{-1}x^{-1}yx=xy[y,x]$, every element of $F$ can be written uniquely in the form $a^kb^sc$ with $0\le k<p$, $0\le s<q$ and $c\in C$.
	\begin{lem}\label{a,b}
		The element $a$ has order $p$ in ${F/N}$, and the element $b$ has order $q$ in ${F/N}$.
	\end{lem}
	\begin{proof}
		Since $a$ and $\gamma(a)$ (see \eqref{gamma}) both have order $p$ and $\beta\circ\alpha=\gamma$, the element $\alpha(a)$ also has order $p$; that is, $a$ has order $p$ in ${F/N}$. Likewise, since $b$ and $\gamma(b)$ have order $q$, so does $\alpha(b)$.
	\end{proof}
	\begin{lem}\label{pq}
		The group ${F/N}$ is a periodic group of period $pq$.
	\end{lem}
	\begin{proof}
		Let $x\in F$ be arbitrary. Since $F/C=\langle a\rangle_p\times \langle b\rangle_q$ is an abelian group of exponent $pq$, we have $\gamma(x^{pq})=1$, that is, $x^{pq}\in C$. As $x^{pq}=(x^p)^q$, it follows from \eqref{N} that $x^{pq}\in N$. Hence the order of the image of $x$ in $F/N$ divides $pq$.
	\end{proof}
	\begin{lem}\label{ab}
		The element $ab$ has order $pq$ in ${F/N}$.
	\end{lem}
	\begin{proof} By Lemma \ref{pq}, the order of $ab$ in $F/N$ divides $pq$. We show that $(ab)^{p}\neq1$ and $(ab)^{q}\neq1$ in $F/N$. Indeed, if $(ab)^{p}=1$ in $F/N$, then $(ab)^{p}\in N\subset C$, and therefore
		$$1=\gamma\big((ab)^{p}\big)=\gamma(a)^p\gamma(b)^p=\gamma(b)^p$$
		in the abelian group $F/C=\langle a\rangle_p\times \langle b\rangle_q$. This contradicts the fact that $\gamma(b)$ has order $q\neq p$.
		
		Similarly, if $(ab)^{q}=1$ in $F/N$, then $(ab)^{q}\in C$, and hence $1=\gamma\big((ab)^{q}\big)=\gamma(a)^q\gamma(b)^q=\gamma(a)^q$, which contradicts the fact that $\gamma(a)$ has order $p\neq q$.
		
		Since $p$ and $q$ are distinct primes and the order of $ab$ divides $pq$, it follows that this order equals $pq$.
	\end{proof}
	\begin{lem}\label{acap}
		$\langle a\rangle\cap \langle b\rangle^{F/N}=\{1\}$, where $\langle b\rangle^{F/N}$ is the normal closure of the element $b$ in ${F/N}$.
	\end{lem}
	\begin{proof} Consider the sequence of canonical epimorphisms $$\langle a\rangle_p\bigcirc^q \langle b\rangle_q \xrightarrow{\beta}\langle a\rangle_p\times \langle b\rangle_q\xrightarrow{\delta}\langle a\rangle_p.$$
		On the one hand, $\delta(\beta(a))=a\neq1$, so $a\notin \mathrm{Ker}(\delta\circ\beta)$, and since $a$ has prime order $p$ in $F/N$ (Lemma \ref{a,b}), we get $\langle a\rangle\cap \mathrm{Ker}(\delta\circ\beta)=\{1\}$. On the other hand, $\delta(\beta(b))=1$, and hence $\langle b\rangle^{F/N}\subset \mathrm{Ker}(\delta\circ\beta)$. Therefore
		$$\langle a\rangle\cap \langle b\rangle^{F/N}\subset \langle a\rangle\cap \mathrm{Ker}(\delta\circ\beta)=\{1\}.$$
	\end{proof}
	\begin{lem}\label{bFN} The normal closure $\langle b\rangle^{F/N}$ of the element $b$ in the group ${F/N}$ is a periodic group of period $q$.
	\end{lem}
	\begin{proof}
		Let the word $x=a^kb^sc$ of $F$ represent an element of the normal closure $\langle b\rangle^{F/N}$, where $0\le k<p$, $0\le s<q$ and $c\in C$. Clearly $b^s, c\in \langle b\rangle^{F/N}$ (the Cartesian subgroup $C$ is generated by the commutators $[a^i,b^j]$, which lie in the normal closure of $b$). Hence the condition $a^kb^sc\in \langle b\rangle^{F/N}$ implies that $a^k\in \langle b\rangle^{F/N}$. By Lemma \ref{acap}, $a^k=1$ in $F/N$, and since $a$ has order $p$ in $F/N$ (Lemma \ref{a,b}) and $0\le k<p$, we obtain $k=0$, so $x=b^sc$. Then $\gamma(x^q)=\gamma(b)^{sq}=1$, that is, $x^q\in C$, and therefore $x^q\in N$ by \eqref{N}. Thus the order of $x$ in $F/N$ divides $q$.
	\end{proof}
	\section{Centralizers of elements of the product $\langle a\rangle_p\bigcirc^q \langle b\rangle_q$}\label{p4}
	Some properties of the strictly Burnside $n$-product $\underset{\mu\in I}{\bigcirc}^nG_\mu$ in the class of groups without involutions are analogous to those of S.~I.~Adian's $n$-periodic products established in \cite{A76} (see also \cite{AA}). Proving them requires a graded presentation and an inductive definition of $\underset{\mu\in I}{\bigcirc}^nG_\mu$ (see \cite[\S3]{I} or \cite[p.~401]{O}). To avoid cluttering the text, we omit this inductive definition and record only the following property, established in \cite[\S3]{I} and \cite[p.~403]{O}.
	
	\begin{lem}\label{IO} Let $n>10^{10}$ be odd, and let $G_\mu$, $\mu\in I$, be a family of groups without involutions. In the strictly Burnside $n$-product $\underset{\mu\in I}{\bigcirc}^nG_\mu$, the centralizer of any nontrivial element $X$ of $G_\mu$ lies in $G_\mu$.
	\end{lem}
	Since $p\ge3$ and $q>10^{10}$ are odd primes, the groups $\langle a\rangle_p$ and $\langle b\rangle_q$ have no involutions, and Lemma \ref{IO} applies to $\langle a\rangle_p\bigcirc^q \langle b\rangle_q$. As every nontrivial element of $\langle a\rangle$ (respectively, $\langle b\rangle$) generates this cyclic group, we obtain
	\begin{lem}\label{1.2} In $\langle a\rangle_p\bigcirc^q \langle b\rangle_q$, the centralizer of every nontrivial element of $\langle a\rangle$ coincides with the cyclic group $\langle a\rangle$, and the centralizer of every nontrivial element of $\langle b\rangle$ coincides with the cyclic group $\langle b\rangle$.
	\end{lem}
	\section{The inner automorphism $i_a$ of the group $\langle a\rangle_p\bigcirc^q \langle b\rangle_q$}\label{p5}
	\begin{lem}\label{reg}
		The inner automorphism $i_a$ of the group ${F/N}$ is a regular automorphism of order $p$ of the subgroup $\langle b\rangle^{F/N}$.
	\end{lem}
	\begin{proof}
		By Lemma \ref{1.2}, the centralizer of $a$ in $F/N=\langle a\rangle_p\bigcirc^q \langle b\rangle_q$ is the cyclic group $\langle a\rangle$. Since $\langle a\rangle\cap\langle b\rangle^{F/N}=\{1\}$ by Lemma \ref{acap}, we have $i_a(x)=a^{-1}xa\neq x$ for every nontrivial $x\in \langle b\rangle^{F/N}$.
		
		It remains to determine the order of the restriction of $i_a$ to $\langle b\rangle^{F/N}$. Since $a^p=1$, we have $i_a^p=i_{a^p}=1$. Suppose that $i_a^k=i_{a^k}$ acts trivially on $\langle b\rangle^{F/N}$ for some $0<k<p$. Then $a^k$ centralizes $b$, so $a^k\in\langle b\rangle$ by Lemma \ref{1.2}. Hence $a^k\in\langle a\rangle\cap\langle b\rangle^{F/N}=\{1\}$ (Lemma \ref{acap}), which is impossible because $a$ has order $p$ (Lemma \ref{a,b}). Therefore the restriction of $i_a$ to $\langle b\rangle^{F/N}$ has order exactly $p$.
	\end{proof}
	
	\begin{lem}\label{ns}
		The inner automorphism $i_a$ of the group ${F/N}$ is not a splitting automorphism of the subgroup $\langle b\rangle^{F/N}$.
	\end{lem}
	\begin{proof}
		If $i_a$ were a splitting automorphism, then the identity
		\[
		g\,g^{i_a}g^{i_a^2}\cdots g^{i_a^{p-1}}=1
		\]
		would hold for every $g\in \langle b\rangle^{F/N}$. In particular, for $g=b^{-1}$ we would obtain
		\[
		{(b^{-1})}\,{(b^{-1})}^{i_a}{(b^{-1})}^{i_a^2}\cdots {(b^{-1})}^{i_a^{p-1}}=({b^{-1}}a^{-1})^pa^{p}=({b^{-1}}a^{-1})^p=1,
		\]
		which contradicts Lemma \ref{ab}, according to which the element ${b^{-1}}a^{-1}=(ab)^{-1}$ has order $pq$.
	\end{proof}
	
	\section{Proof of Theorem \ref{th}}\label{p6}
	
	Let $G={F/N}$; this group is generated by the images of $a$ and $b$. By Lemma \ref{pq}, $G$ is periodic of period $pq$, and by Lemma \ref{ab} it contains an element of order $pq$. By Lemma \ref{a,b}, the generator $a$ has order $p$, which proves item~1. By Lemma \ref{bFN}, the normal closure $\langle b\rangle^{G}$ is a periodic group of period $q$, which proves item~2. By Lemma \ref{reg}, the inner automorphism $i_a$ acts on $\langle b\rangle^{G}$ as a regular automorphism of order $p$, and by Lemma \ref{ns} it is not a splitting automorphism of $\langle b\rangle^{G}$; this proves item~3 and completes the proof of the theorem.
	
	Since $\langle b\rangle^{G}$ has period $q\neq p$, it contains no nontrivial $p$-elements. Thus the restriction of $i_a$ to $\langle b\rangle^{G}$ is a regular automorphism of order $p$ of a periodic $p'$-group that is not splitting, which gives a negative answer to Mazurov's question.

\end{document}